\documentclass[12pt]{article}

\usepackage{amsfonts,amssymb,stmaryrd,amsthm,amsmath}
\usepackage{frcursive}

\usepackage{tikz}
\usetikzlibrary{matrix,arrows,decorations.markings,cd}

\renewcommand{\theenumi}{(\roman{enumi})}
\renewcommand{\labelenumi}{\theenumi}
\renewcommand{\theenumii}{(\alph{enumii})}
\renewcommand{\labelenumii}{\theenumii}

\newcommand\longto{{\longrightarrow}}

\newcommand\RR{\mathbb{R}}
\newcommand\CC{\mathbb{C}}\newcommand\QQ{\mathbb{Q}}
\newcommand\NN{\mathbb{N}}\newcommand\ZZ{\mathbb{Z}}

\newcommand\inv{{^{-1}}}
\newcommand\bkprod{{\odot_0}}
\newcommand\uo{{\underline{o}}}

\newcommand\Supp{{\operatorname{Supp}}}
\renewcommand\div{{\operatorname{div}}}
\newcommand\Ker{{\operatorname{Ker}}}
\newcommand\Pic{{\operatorname{Pic}}}
\newcommand\rk{{\operatorname{rk}}}

\newcommand\Oc{{\mathcal O}}
\newcommand\Bcal{{\mathcal B}}
\newcommand\Li{{\mathcal L}}

\newcommand\Lie{{\operatorname{Lie}}}
\newcommand\Inf{{\operatorname{Inf}}}
\newcommand\supp{{\operatorname{supp}}}

\newcommand\LR{{\operatorname{LR}}}
\newcommand\cLR{{\mathcal{LR}}}

\renewcommand{\lg}{{\mathfrak g}}

\newcommand{\lb}{{\mathfrak b}}
\newcommand{\lieu}{{\mathfrak u}}\newcommand{\liet}{{\mathfrak t}}

\newcommand\Ho{\operatorname{H}}

\usepackage[utf8]{inputenc}  

\usepackage[T1]{fontenc}

\usepackage{hyperref}

\newtheorem{theo}{Theorem}
\newtheorem{prop}[theo]{Proposition}
\newtheorem{lemma}[theo]{Lemma}
\newtheorem{coro}[theo]{Corollary}

\theoremstyle{definition}
\newtheorem{remark}{Remark}
\newtheorem{exple}{Example}
\newtheorem{remarks}{Remarks}
\newcommand\DynkinNodeSize{2mm}
\newcommand\DynkinArrowLength{3mm}
\tikzset{
  dnode/.style={
    circle,
    inner sep=0pt,
    minimum size=\DynkinNodeSize,
    fill=white,
    draw},
  middlearrow/.style={
    decoration={markings,
      mark=at position 0.6 with
      {\draw (0:0mm) -- +(+135:\DynkinArrowLength); \draw (0:0mm) -- +(-135:\DynkinArrowLength);},
    },
    postaction={decorate}
  },
  leftrightarrow/.style={
    decoration={markings,
      mark=at position 0.999 with
      {
      \draw (0:0mm) -- +(+135:\DynkinArrowLength); \draw (0:0mm) -- +(-135:\DynkinArrowLength);
      },
      mark=at position 0.001 with
      {
      \draw (0:0mm) -- +(+45:\DynkinArrowLength); \draw (0:0mm) -- +(-45:\DynkinArrowLength);
      },
    },
    postaction={decorate}
  },
  sedge/.style={
  },
  dedge/.style={
    middlearrow,
    double distance=0.5mm,
  },
  tedge/.style={
    middlearrow,
    double distance=1.0mm+\pgflinewidth,
    postaction={draw}, 
  },
  infedge/.style={
    leftrightarrow,
    double distance=0.5mm,
  }
}

\begin{document}

\title{Intersection multiplicity one for the Belkale-Kumar product in $G/B$}
\author{Luca Francone and Nicolas Ressayre}

\maketitle

\begin{abstract}
Consider the complete flag variety $X$ of a
  complex semisimple algebraic group $G$. 
We show that the structure coefficients of the Belkale-Kumar product
$\odot_0$, on the cohomology ${\operatorname H}^{*}(X,\ZZ)$,
are all either $0$ or $1$.
We also derive some consequences.
The proof that is mainly geometric also uses new combinatorial results on root systems. 
Moreover, it is uniform and avoids case by case considerations.
\end{abstract}



\section{Introduction}

Let $G$ be a complex semisimple group and let $B$ be a Borel subgroup
of $G$. 
In this paper, we are interested in the Belkale-Kumar product $\bkprod$
on the cohomology group of the complete flag variety $G/B$.

Let $\uo$ denote the base point $B/B$ of $G/B$.
Fix a maximal torus $T$ of $B$ and let $W$ be the Weyl group of $G$.
For any $w\in W$, let $X_w=\overline{Bw\uo}$ be the corresponding Schubert variety and let $[X_w]\in \Ho^*(G/B,\CC)$ be its 
cohomology cycle.
Then, $([X_{w}])_{w\in W}$ is a basis for the cohomology group $\Ho^{*}(G/B,\ZZ)$.
  The structure coefficients $c_{uv}^w$ of the cup product are written as
  \begin{eqnarray}
    \label{eq:defc}
    [X_u]\cdot[X_v]=\sum_{w\in W} c_{uv}^w[X_w].
  \end{eqnarray}

  Let $\Phi$ be the set of  roots
  of $G$. 
  We denote by $\Phi^+$, $\Phi^-$ and $\Delta$  the set of positive,
  negative  and simple roots  corresponding to $B$, respectively.
 For $w\in W$, we denote by  $\Phi(w)=\Phi^+\cap w\inv\Phi^-$ the set of
 inversions of $w$.  
 For its applications to the geometry of the eigencone, Belkale-Kumar
 defined in \cite{BK} a new product $\bkprod$ on  $\Ho^*(G/P,\CC)$, for
 any parabolic subgroup $P$. 
  When $P=B$, the structure constants $\tilde c_{uv}^w$ of the Belkale-Kumar product,
\begin{eqnarray}
    \label{eq:defctilde1}
    [X_u]\bkprod [X_v]=\sum_{w\in W}\tilde c_{uv}^w[X_w]
  \end{eqnarray}
can be defined as follows (see \cite[Corollary~44]{BK}):
\begin{eqnarray}
  \label{eq:defctilde2}
  \tilde c_{uv}^w=\left\{
    \begin{array}{ll}
      c_{uv}^w &\mbox{ if } \Phi(u)\cap\Phi(v)=\Phi(w) \mbox{ and }\Phi(u)\cup\Phi(v)=\Phi^+,\\
0 &\mbox{ otherwise.}
    \end{array}
\right .
\end{eqnarray}
The product $\bkprod$ is  associative and satisfies Poincar\'e duality. 
Our main result can be stated as follows.

\begin{theo}
  \label{th:main}
  Let $u,v$ and $w$ in $W$ be such that $\Phi(u)\cap\Phi(v)=\Phi(w)$ and
  $\Phi(u)\cup\Phi(v)=\Phi^+$. Then
  $$
c_{uv}^w=1.
$$
\end{theo}

In Section~\ref{sec:consequences} we state some consequences of Theorem~\ref{th:main} that concern: 
the Bruhat order of $W$, the number of descents of Weyl group elements,  the geometry of the eigencone and the
cohomological components of the tensor product decomposition of irreducible $G$-modules.
The proofs of these results are either trivial or carried out in Section~\ref{sec:pf coro}.

\medskip

Theorem~\ref{th:main} was conjectured by Belkale-Kumar in oral
discussions since 2006 and is stated as a question in \cite[Question~1]{DR:prv}.
A lot of special cases were known before. In \cite[Corollary~4]{Rich:mult},
E. Richmond proved it in type A. As noticed in \cite{Richmond:recursion}
or \cite[Corollary~1]{multi}, Richmond's
proof also works in type C. Type~B is proved in
\cite[Proposition~16]{distrib}. In \cite{DR:mult1}, all the classical
types are solved. The cases $G_2$, $F_4$ and
$E_6$ can be checked using a computer. 
Note finally that, in \cite[Conjecture~1]{distrib}, a conjecture for any homogeneous
space $G/P$ is formulated to extend Theorem~\ref{th:main} to any
homogeneous space $G/P$.

\medskip

Our proof of Theorem~\ref{th:main} is uniform on the type and can be divided in two main parts.
The first one uses geometric arguments to reformulate Theorem~\ref{th:main} as a linear algebra statement.
This part mainly relies on properties of birational morphisms and on the fact that complete flag varieties are simply
connected  (see Section~\ref{sec:strategy}).
The second part consists in the proof of the aforementioned reformulation.
It takes advantage of Theorem~\ref{th:combi}, which is a result on the combinatorics of root systems of independent interest. 
In order to state Theorem~\ref{th:combi}, we now introduce some notation. 

\medskip


For $\varphi$ and $\psi$ in $\Phi$, we write $\varphi \leq \psi$ if
$\psi-\varphi\in \sum_{\alpha\in\Delta}\NN\alpha$.
 If $\varphi \leq \psi$ and
$\varphi\neq\psi$, we write $\varphi < \psi$. 
We set $[\varphi;\psi]:=\{\gamma\in\Phi\;:\; \varphi\leq\gamma\leq
\psi\}$.
A subset $\Phi_1$ of $\Phi^+$ is said to be {\it biconvex} if
it is of the form $\Phi(w)$ for some  $w \in W$.

\begin{theo}
  \label{th:combi}
  Let $\Phi_1, \Phi_2$ and $\Phi_3$ be three biconvex subsets of $\Phi^+$ such that $\Phi_3=\Phi_1\cup\Phi_2$ and $\Phi_1 \cap \Phi_2= \emptyset$.
  Let $\beta$ and $\gamma$ be two positive roots such that
  \begin{enumerate}
  \item $\beta\in\Phi_1$;
    \item $\gamma\not\in \Phi_3$;
    \item $\gamma+ \beta\in\Phi_3$.
    \end{enumerate}
    Then $\Phi_2\cap [\beta;\gamma]$ is empty.
 \end{theo}

The proof of Theorem~\ref{th:combi} is done in Section~\ref{sec:proof thm combi} 
and relies on some recent results on \textit{quotient root systems} \cite{Dim:quotroot}.

\medskip

{\bf Acknowledgements.} 
We are grateful to Stephane Druel and Christophe Hohlweg for useful discussions.
We also thank an anonymous referee for pointing out to us that the results of \cite{Dim:quotroot} 
could provide a substantial simplification of our original proof of Theorem~\ref{th:combi}.

\tableofcontents

\section{Some consequences of Theorem~\ref{th:main}}
\label{sec:consequences}

From now on, we are in the setting of the introduction. 
Most of the results stated in this section are proved in Section~\ref{sec:pf coro}. 
We start by introducing some notation that is freely used through the paper.

\medskip

\noindent \textbf{Notation.}
We denote by $w_0$ the longest element of $W$.
For $w \in W$, we set $w^\vee=w_0w$ the Poincaré dual of $w$, so that $\Phi(w^\vee)= \Phi^+ \setminus \Phi(w)$.

Given three subsets $\Phi_1$, $\Phi_2$  and $\Phi_3$ in $\Phi^+$, we
write $\Phi_1\sqcap\Phi_2=\Phi_3$ if $\Phi_1\cap\Phi_2=\Phi_3$ and
$\Phi_1\cup\Phi_2=\Phi^+$.
Similarly, we write $\Phi_3=\Phi_1\sqcup\Phi_2$ if
$\Phi_3=\Phi_1\cup\Phi_2$ and $\Phi_1\cap\Phi_2=\emptyset$. 

A subset $\Phi_1$ of $\Phi^+$ is said to be {\it convex} if, for any
$\varphi,\psi\in\Phi_1$ such that $\varphi+\psi\in\Phi$, we have
$\varphi+\psi\in\Phi_1$.
It is said to be {\it coconvex} if its complementary $\Phi^+\backslash\Phi_1$
is convex; and it is said to be {\it biconvex} if it is convex and
coconvex.
By \cite[Proposition~5.10]{Kostant:harmform1}, this definition of biconvex sets coincides with the one given in the introduction.

\medskip

Let $u$, $v$ and $w$ as in Theorem~\ref{th:main}.
To emphasize the symmetry in $u$, $v$ and $w^\vee$, we set
$$
w_1=w^\vee\quad w_2=u\quad w_3=v.
$$
The assumptions $\Phi(u)\cap\Phi(v)=\Phi(w)$ and
$\Phi(u)\cup\Phi(v)=\Phi^+$ can be translated as
\begin{equation}
  \label{eq:hypwi}
  \Phi^+=\Phi(w_1^\vee)\sqcup \Phi(w_2^\vee)\sqcup \Phi(w_3^\vee).
\end{equation}

We denote by $\leq$ the Bruhat order on $W$: for $v, w\in W$,  $v\leq w$ if and only if $X_v\subset
X_w$.

\subsection{On the Bruhat order}

\begin{coro}
  \label{cor:Bruhatorder}
  Let $w_1$, $w_2$ and $w_3$ in $W$. If 
  Condition~\eqref{eq:hypwi} holds, then the only element $x\in W$ such
  that  $w_ix\leq w_i$, for $i=1,2$ and $3$ is the neutral element $x=e$.
\end{coro}

\subsection{The number of descents}

For $w\in W$, denote by  $\ell(w)$ the cardinality of $\Phi(w)$; it is
the length of $w$ when $W$ is thought as a Coxeter group.
For $w\in W$, we consider the set of left descents:
$$
D(w)=\{\alpha\in \Delta\,:\,\ell(s_\alpha u)<\ell(u)\}
$$
and denote by $d(w)$ the cardinality of $D(w)$. 

\begin{coro}\label{cor:D}
  Let $w_1,w_2$ and $w_3$ in $W$ satisfying \eqref{eq:hypwi}, then
$$
d(w_1)+d(w_2)+d(w_3)=2\rk(G)
\quad{\rm and}\quad
d(w_1^\vee)+d(w_2^\vee)+d(w_3^\vee)=\rk(G).
$$
\end{coro}

Based on some computations with a computer we ask the following question: under the assumptions of Corollary~\ref{cor:D}, do we have
$$
d(w_1^\vee)+d(w_2^\vee)  = d(w_1w_2^{-1}) {\rm\ or\ } d(w_2w_1^{-1})?
$$
This have been checked for any root system of rank at most 5 (see the
source code on
\cite{mapage}).

\subsection{\texorpdfstring{Using a Belkale-Kumar expression of $c_{uv}^w$}{Using a Belkale-Kumar expression of cuvw}}

Let $B^-$ be the opposite Borel subgroup of $B$, so that $B\cap B^-=T$.
Let $U$ (also denoted by  $U^+$) and $U^-$  be respectively the unipotent radical of $B$ and $B^-$. 
In \cite[Theorem 43]{BK} Belkale and Kumar give an isomorphism of graded rings:
$$
\phi\colon \left(\Ho^{*}(G/B,\CC),\bkprod\right) \cong \left[\Ho^{*}(\lieu^{+})\otimes \Ho^{*}(\lieu^{-})\right]^{\liet},
$$
where $\lieu^{\pm}= \Lie(U^{\pm})$, $\liet= \Lie(T)$ and
$\Ho^{*}(\lieu^{\pm})$ denotes the Lie algebra cohomology of the nilpotent algebras $\lieu^{\pm}$.
They derive in \cite[Corollary 44-({\em ii})]{BK} an expression for
the coefficients $\tilde c_{uv}^w$. Using it, we get:
\begin{coro}
  If $\Phi(w)=\Phi(w_1)\sqcup \Phi(w_2)$, then
  $$
  \prod_{\alpha\in \Phi(w^{-1})} \langle \rho,\alpha\rangle =
  \left(\prod_{\alpha\in \Phi(w_1^{-1})} \langle \rho,\alpha\rangle
  \right) \left(\prod_{\alpha\in \Phi(w_2^{-1})} \langle
    \rho,\alpha\rangle \right),
  $$
  where $\rho$ is one-half the sum of the positive roots and $\langle\cdot,\cdot\rangle$ the Killing form. 
\end{coro}

\subsection{Minimal regular faces of the eigencone}

Let $X(T)^+$ (resp. $X(T)^{++}$) denote the set of dominant
(resp. strictly dominant) characters of $T$ (relatively
to $B$). For $\lambda\in X(T)^+$, we denote by $V(\lambda)$ the
irreducible $G$-module of highest weight $\lambda$. 
For a $G$-module $V$, we denote by $V^G$ the set of $G$-invariant vectors.
Set
$$
\LR(G)=\{(\lambda_1,\lambda_2,\lambda_3)\in (X(T)^+)^3\;:\;
(V(\lambda_1)\otimes V(\lambda_2)\otimes V(\lambda_3))^G\neq\{0\}\ \}.
$$
This set is known to be a finitely generated semigroup (see
e.g. \cite{Kumar:surveyEMS}).
The convex cone $\cLR(G)$ generated by $\LR(G)$ in
$(X(T)\otimes\QQ)^3$ is closed and polyhedral. A face of $\cLR(G)$ is
said to be regular if it intersects $(X(T)^{++})^3$.
By \cite{GITEigen}, the regular faces are controlled by the Belkale-Kumar
product on $H^*(G/P,\ZZ)$ for various standard parabolic subgroups $P$ of
$G$.

\begin{coro}\label{cor:regface}
  Assume that $G$ is semi-simple and simply connected. 
  Let $w_1$, $w_2$ and $w_3$ in $W$ satisfying
  Condition~\eqref{eq:hypwi}. Then
  $${\mathcal F}_{(w_1,w_2,w_3)}=\{(\lambda_1,\lambda_2,\lambda_3)\in (X(T)^+)^3\;:\;w_1\inv
  \lambda_1+w_2\inv\lambda_2+w_2\inv\lambda_3=0\}$$
  is the set of points in  $\LR(G)$ that belong to a regular face of
  $\cLR(G)$.
  As a semigroup, ${\mathcal F}_{(w_1,w_2,w_3)}$ is freely generated
 by $2\rk(G)$ elements.
 Moreover, any codimension $\rk(G)$
  regular face is obtained in such a way.
\end{coro}

\begin{remarks}
  \begin{enumerate}
  \item A significant part of Corollary~\ref{cor:regface} (which is
    even equivalent to Theorem~\ref{th:main}) is that there exists
    regular weights $\lambda_1, \lambda_2$ and $\lambda_3$ in $X(T)^{++}$
    such that $w_1\inv \lambda_1+w_2\inv\lambda_2+w_2\inv\lambda_3=0$.
\item The $2\rk(G)$ generators of the semigroup ${\mathcal
    F}_{(w_1,w_2,w_3)}$ are described geometrically in the proof of
  the corollary (see Section~\ref{sec:pf coro}) as the line bundles
  ${\mathcal O} (D_i)$ associated to some explicit divisors $D_i$.
  The triple of weights $(\lambda_1,\lambda_2,\lambda_3)$
  corresponding to 
  ${\mathcal O}(D_i)$ can be derived from \cite[Theorem~8]{BK:rays}.
  \end{enumerate}

\end{remarks}

\subsection{Cohomological components of tensor products}

For $\lambda\in X(T)$, we denote by $\Li(\lambda)$ the $G$-linearized line bundle on $G/B$ such 
that $B$ acts on the fiber over $\uo$ by the character $-\lambda$.
Recall that $\uo$ denotes the base point $B/B$ of $G/B$.
If $\lambda$ is dominant, the Borel-Weil theorem asserts that the space
of sections ${\rm H}^0(G/B,\Li(\lambda))$ is
isomorphic to $V(\lambda)^*$ as a representation of  $G$.
We also set $$\lambda^*=-w_0\lambda.$$
 The points $(\lambda_1,\,\lambda_2, \lambda_1^* + \lambda_2^*)$ (for
 $\lambda_1,\lambda_2\in X(T)^+$) of $\LR(G)$  have the following 
geometric property:
the morphism 
\begin{equation}
  \label{eq:cupzero}
   {\rm H}^0(G/B,\Li(\lambda_1))\otimes  {\rm H}^0(G/B,\Li(\lambda_2))\longto  
{\rm H}^0(G/B,\Li(\lambda_1+\lambda_2)),
\end{equation}
given by the product of sections is nonzero.

Following Dimitrov-Roth (see~\cite{DR:prv1,DR:prv2}), 
we introduce a natural generalization of 
these points of $\LR(G)$ coming from the Borel-Weil-Bott theorem.
For $w\in W$ and $\lambda\in X(T)$, set:
\begin{equation}
  \label{eq:Aaff}
  w\cdot\lambda=w(\lambda+\rho)-\rho.
\end{equation}
The Borel-Weil-Bott theorem asserts that, for any dominant weight $\lambda$ and
any $w\in W$, the $G$-module ${\rm H}^{\ell(w)}(G/B,\Li(w\cdot\lambda))$ is
isomorphic to $V(\lambda)^*$.
Let $(\lambda_1,\,\lambda_2,\,\lambda_3)$ be a triple of dominant weights.
We say that $(\lambda_1,\,\lambda_2,\,\lambda_3^*)$ is a 
{\it cohomological point} of $\LR(G)$ if the cup product:
\begin{equation}
  \label{eq:cup}
   {\rm H}^{\ell(w_1)}(G/B,\Li(w_1\cdot\lambda_1))\otimes  
{\rm H}^{\ell(w_2)}(G/B,\Li({w_2\cdot\lambda_2}))\longto  
{\rm H}^{\ell(w_3^\vee)}(G/B,\Li({w_3^\vee\cdot\lambda_3}))
\end{equation}
is nonzero for some  $w_1,w_2,w_3\in W$. This implies in particular
that   $\ell(w_3^\vee)=\ell(w_1)+\ell(w_2)$ and
$w_1\cdot\lambda_1+w_2\cdot\lambda_2=w_3^\vee\cdot\lambda_3$.

\begin{theo}[Dimitrov-Roth]
  \label{th:DR}
  Let $w_1$, $w_2$, $w_3$ in $W$ and
  $(\mu_1,\mu_2,\mu_3)\in (X(T)^+)^3$ such that
  \begin{enumerate}
  \item $\ell(w_3)=\ell(w_1)+\ell(w_2)$;
  \item $\mu_3=\mu_1+\mu_2$;
    \item $w_i\cdot\mu_i$ is dominant for $i=1,2,3$.
  \end{enumerate}
  Then the  cup product map  
\begin{eqnarray}
  \label{eq:cup2}
   {\rm H}^{\ell(w_1)}(G/B,\Li(\mu_1))\otimes  
{\rm H}^{\ell(w_2)}(G/B,\Li(\mu_2))\longto  
{\rm H}^{\ell(w_3)}(G/B,\Li(\mu_3)),
  \end{eqnarray}
  is nonzero if and only if $\Phi(w_3)=\Phi(w_1)\sqcup\Phi(w_2)$.
\end{theo}

Under the assumption of Theorem~\ref{th:DR} and $\Phi(w_3)=\Phi(w_1)\sqcup\Phi(w_2)$, set
$\lambda_i=w_i\cdot\mu_i$ for $i=1,2,3$.
By the Borel-Weil-Bott theorem, Theorem~\ref{th:DR} gives a surjective
map
$$
V(\lambda_1)^*\otimes V(\lambda_2)^*\longto V(\lambda_3)^*.
$$
In particular the point $(\lambda_1,\lambda_2,\lambda_3^*)$ belongs
to $\LR(G)$.

On the other hand, the condition $\mu_3=\mu_1+\mu_2$ is equivalent to
$$
w_1\inv\cdot\lambda_1+w_2\inv\cdot\lambda_2=w_3\inv\cdot\lambda_3,
$$
which is also equivalent to
\begin{equation}
w_1\inv\lambda_1+w_2\inv\lambda_2=w_3\inv\lambda_3.\label{eq:32}
\end{equation}
Indeed, using that  $\Phi(w_3)=\Phi(w_1)\sqcup\Phi(w_2)$ and $\rho=\frac 1
2\sum_{\beta\in\Phi^+}\beta$, one easily deduces  
$$
\rho=w_1\inv\rho+w_2\inv\rho-w_3\inv\rho.$$
 In particular, from Theorem~\ref{th:main}, equation~\eqref{eq:32} and
Corollary~\ref{cor:regface} we deduce the following Corollary.

\begin{coro}
  \label{cor:cohomcomp}
  The point $(\lambda_1,\lambda_2,\lambda_3^*)\in X(T)^3$ is a cohomological point of $\LR(G)$ if and only if it belongs to a regular face of
  codimension $\rk(G)$ of $\cLR(G)$.
\end{coro}

\section{Overview of the proof of Theorem~\ref{th:main}}

For the reader's convenience, we explain in this section the steps of
the demonstration. In order to identify the guiding ideas more
clearly, we will define the various objects precisely later in the article.
Fix $u,v,w \in W$ as in the statement of Theorem~\ref{th:main}.

\renewcommand{\theenumii}{(\roman{enumii})}
\renewcommand{\labelenumii}{\theenumii}
\renewcommand{\theenumi}{\Alph{enumi}.}
\renewcommand{\labelenumi}{\theenumi}

\begin{enumerate}
\item We consider an incidence projective $G$-variety $Y$ endowed with
  an equivariant morphism $\eta\,:\, Y\longto (G/B)^3$.  The first
  observation (see~Proposition~\ref{prop:eta bir}) is
$$
c_{uv}^w=1\iff \eta\mbox{ is birational.}
$$

\item Then, using an argument of topology, we prove that $\eta$ is
birational if and only if
\begin{center}
  the ramification divisor $R_\eta$ of $\eta$ is contracted.
\end{center}
jSee Section \ref{sec:bir} for more details.

\item The next step consists in describing explicitly the irreducible
components of the divisor $R_\eta$. In particular, we parametrize these components by combinatorial objects. 
Fix such a component $D$.
By the previous point, it remains to show that $D$ is contracted.

\item
 By a general fact of complex algebraic geometry, $D$ is contracted if and only if for a general $x\in D$ 
 the Kernel of the tangent map $T_x\eta$ intersects non-trivially the tangent space $T_x D$.

\item 
Recall that $U$ is the unipotent radical of $B$. 
We define an explicit point $x_D$ of $D$. 
By an argument of  $G$-equivariance, 
we observe that it is sufficient to check the infinitesimal  criterium of Step D on points $x \in D$ of the form
 $x=(g,g') \cdot x_D$  $(g,g'\in U)$. 
 Therefore, we need to show that
$$
\forall g,g'\in U, \qquad\Ker(T_{(g,g')x_0}\eta)\cap
T_{(g,g')x_0}D\neq\{0\}.
$$

\item
Next, we define a matrix $M(g,g')$ depending on $g,g' \in U$ (and on the combinatorial data corresponding to $D$) 
such that
$$
\Ker(M(g,g'))=\Ker(T_{(g_1,g_2)x_0}\eta)\cap T_{(g_1,g_2)x_0}D.
$$

\item
Using a grading by weights, we show that many entries of the
matrix $M(g,g')$ are $0$ (see Lemma~\ref{lem:actionU}). Now, two cases occur:

\medskip

\noindent\underline{Case 1}.
The matrix $M(g,g')$ contains a diagonal block with more columns than rows. Thus,
the Kernel of $M(g,g')$ is not trivial, ending the proof in this case.

\medskip

\noindent\underline{Case 2}.
If case 1 does not occur, the argument is more subtle. 
Analysing the shape of the matrix $M(g,g')$, we
define a family $M_1(g,g'), \dots, M_s(g,g')$ of square submatrices such that $M(g,g')$ is
not injective if and only if at least one of the $M_i(g,g')$ is not
invertible.

Using Theorem~\ref{th:combi}, we prove that
the determinant of each $M_i(g,g')$ only depends  on either $g$ or
$g'$.
Then using the fact that Theorem~\ref{th:main} holds in the
trivial case of Poincaré duality, we deduce that at least one of the
$\det(M_i(g,g'))$ is zero, without computing these determinants. This completes the argument.
\end{enumerate}

\renewcommand{\theenumi}{(\roman{enumi})}
\renewcommand{\labelenumi}{\theenumi}
\renewcommand{\theenumii}{(\alph{enumii})}
\renewcommand{\labelenumii}{\theenumii}

\section{The geometric strategy}
\label{sec:strategy}

We now start the proof of Theorem~\ref{th:main}.

\subsection{Incidence variety}
\label{sec:incidence}

Recall that for any $w \in W$, $[X_{w^\vee}]$ is the Poincaré dual of $[X_w].$
Since $\Ho^*(G/B,\CC)$ is graded, if  $c_{uv}^w\neq 0$ then
\begin{equation}
\ell(u)+\ell(v)=\ell(w)+\ell(w_0).\label{eq:length}
\end{equation}
Assuming \eqref{eq:length}, by Kleiman's theorem we have that $c_{uv}^w$ is the
cardinality of the intersection
$$
g_uX_u\cap g_vX_v\cap g_wX_{w^\vee}
$$
for general $(g_u,g_v,g_w)\in G^3$.

For any $(w_1,w_2,w_3) \in W^3$, we associate the incidence variety $Y=Y(w_1,w_2,w_3)$ defined by
\begin{equation}
Y=\{p=(z,g_1\uo,g_2\uo,g_3\uo)\in (G/B)^4\, : \, z\in g_1X_{w_1}\cap g_2X_{w_2}\cap g_3X_{w_3} \},
\label{eq:defYinc}
\end{equation}
endowed with its projections $\pi\,:\,Y \longto G/B$ and
$\eta\,:\,Y\longto (G/B)^3$ mapping $p$ respectively to $z$ and to $(g_1\uo,g_2\uo,g_3\uo)$.
Notice that, as for the incidence variety $Y$, the maps $\pi$ and $\eta$ depend on the choice of the triple of Weyl group elements $(w_1,w_2,w_3)$.
The integer $c_{uv}^w$ is then interpreted as the cardinality of a general fiber
of $\eta$. 
We get

\begin{prop}
\label{prop:eta bir}
Let $u,v,w\in W$ satisfying~\eqref{eq:length}. 
Then $c_{uv}^w=1$ if and only if the map $\eta$ associated to the triple $(w_1,w_2,w_3)=(u,v,w^\vee)$ is birational.
\end{prop}

\subsection{Incidence variety as a fibered product}

Let $(w_1,w_2,w_3) \in W^3$.
Set $X=(G/B)^3$ and
$$x_0=(w_1\inv \uo, w_2\inv \uo,
w_3\inv \uo)\in X.
$$
Note that $G^3$ acts on $X$.
Set
$C^+=B^3 \cdot x_0$, and denote  by $\bar C^+$ the closure of $C^+$ in $X$.
Notice that 
$$
\bar C^+= X_{w_1^{-1}} \times X_{w_2^{-1}} \times X_{w_3^{-1}}.
$$
The group $B$ acts on $G\times \bar C^+$ by the formula
$b \cdot (g,z)=(gb^\inv,bz)$. The quotient of $G\times \bar C^+$ under this
action is a projective variety denoted by $G\times_B\bar C^+$. The
class of $(g,z)$ is denoted by $[g:z]$.
The map
$$
\begin{array}{cccl}
  \phi\;:&G\times_B\bar C^+&\longto&Y\\
  &[g:(z_1,z_2,z_3)]&\longmapsto&(g\uo,gz_1,gz_2,gz_3)
\end{array}
$$
is an isomorphism. We worn the reader that we freely identify these two varieties by means of the isomorphism $\phi$ through the text.

Observe that, modulo $\phi$, $\eta$ identifies with
$[g:(z_1,z_2,z_3)]\longmapsto (gz_1,gz_2,gz_3)$, and $\pi$ with
$[g:(z_1,z_2,z_3)]\longmapsto g\uo$.

\medskip

We now consider $G\times_B  C^+$ as an open subset of $G\times_B\bar
C^+$ and denote by $\eta^\circ$ the restriction of $\eta$ to this
open set.
Then, $G\times_B  C^+$ identifies (once more, via $\phi$) with
$$
Y^\circ=\{(z,g_1\uo,g_2\uo,g_3\uo)\in (G/B)^4\, : \, z\in g_1X_{w_1}^\circ\cap g_2X_{w_2}^\circ\cap g_3X_{w_3}^\circ \},
$$
where,  for an element $w\in W$, we use the notation $X_w^\circ=Bw\uo$.

\subsection{\texorpdfstring{The tangent map of $\eta$}{The tangent map of eta}}

\begin{prop}
\label{prop:Kereta}  
Let $(w_1,w_2,w_3) \in W^3$ and $p=(z,g_1\uo,g_2\uo,g_3\uo)\in Y(w_1,w_2,w_3)$.
The differential $T_p \pi$ of $\pi$ at  $p$ induces an isomorphism of vector spaces
$$
T_p \pi : \ker (T_p \eta) \longto \,  T_z g_1X_{w_1}\cap T_z g_2X_{w_2}\cap T_z g_3X_{w_3}.
$$
\end{prop}

\begin{proof}
Denote $Y(w_1,w_2,w_2)$ by $Y$. 
Notice that $\ker(T_p\eta)= T_p(Y_{\eta(p)})$, where $Y_{\eta(p)}$ is the scheme theoretic fiber of the point $\eta(p)$ along the morphism $\eta.$
The restriction of the morphism $\pi$ to $Y_{\eta(p)}$ induces an isomorphism between $Y_{\eta(p)}$ and the scheme theoretic intersection $g_1X_{w_1}\cap g_2X_{w_2}\cap g_3X_{w_3} \subset G/B$.
We conclude by noticing that 
$$
T_z \bigl( g_1X_{w_1}\cap g_2X_{w_2}\cap g_3X_{w_3} \bigr)=T_z g_1X_{w_1}\cap T_z g_2X_{w_2}\cap T_z g_3X_{w_3}.
$$
\end{proof}

\subsection{About birational maps}
\label{sec:bir}

Let $f\,:\,Y\longto X$ be a dominant morphism between irreducible
varieties of the same dimension. We say that $f$ is {\it
  generically finite}. 
The {\it degree of $f$} is defined to be $\deg(f)=[\CC(Y):\CC(X)]$.
The degree of $f$ is one, if and only if $f$ is birational.
A prime divisor $D$ of $Y$ is said to be {\it contracted} if the closure 
$\overline{f(D)}$ of its image has codimension at least two.
We use the following consequence of the main Zariski theorem
(see~\cite[Chap III, Section 9, Proposition~1]{Mumford:red}).

\begin{prop}\label{prop:Zariskidivisor}
  Assume that $f$ is birational and that $X$ is normal. Let $D$ be a
  prime divisor of $Y$. Assume that $D$ is not contracted. 
  
  Then,  the restriction of $f$ to $D$ is still birational. Moreover, 
  if  $D'$ is a prime divisor of $Y$ which is not contracted and
    such that $\overline{f(D)}=\overline{f(D')}$, then $D=D'$.
\end{prop}

Come back to a generically finite morphism $f\,:\,Y\longto X$.
Assume in addition that $Y$ is normal and $X$ is smooth.
Let $Y^{\rm reg}$ denote the open set of smooth points of $Y$.
The determinant of the tangent map of
$f$ defines a Cartier divisor
$R_f$ in $Y^{\rm reg}$, called the {\it ramification} divisor.
Taking its closure, we get a Weyl divisor on $Y$, still denoted by
$R_f$.
Recall that $Y-Y^{\rm reg}$ has codimension at least 2. 
Let $\Supp(R_f)$ denote the reduced support of $R_f$.

\begin{prop}
  \label{prop:ramdiv}
Let $f\,:\,Y\longto X$ be a generically finite morphism. 
  Assume, in addition,  that
  \begin{enumerate}
  \item $X$ is smooth and simply connected;
    \item $Y$ is normal and projective;
  \end{enumerate}
Then, $f$ is birational if and only if each irreducible component of $\Supp(R_f)$ is contracted.
\end{prop}

\begin{proof}
  If $f$ is birational, the Zariski's main theorem implies that $\Supp(R_f)$ is contracted.

  For the converse, using the Stein factorization \cite[Corollary~11.5]{Hart}, we may
  assume that $f$ is finite. Then $f$ is a covering from $Y\backslash
  R_f$ onto $X\backslash f(\Supp(R_f))$.
  Since each component of $\Supp(R_f)$ is contracted, the
  fundamental groups of $X$ and $X\backslash f(\Supp(R_f))$ coincide,
  and $X\backslash f(\Supp(R_f))$ is simply connected.
  The proposition follows.
\end{proof}

Let now $(w_1,w_2,w_3)\in W^3$ such that $c_{w_1w_2}^{w_3^\vee}\neq 0$. 
Then, the map $\eta\,:\,Y=Y(w_1,w_2,w_3)\longto (G/B)^3$ satisfies the assumptions of 
Proposition~\ref{prop:ramdiv}. We get:

\begin{prop}
  \label{prop:c=1contract}
  The map $\eta$ is birational if and only if $R_\eta$ is contracted.
\end{prop}

\section{\texorpdfstring{First properties of the map $\eta$}{First properties of the map eta}}
\label{sec:etaleopen}

\subsection{Bruhat order}

For later use, we fix some notation on the Bruhat order.
The Bruhat order is generated by the covering relations:
$v\leq^1 w$. 
This condition means that $X_v\subset X_w$ and
$\dim(X_w)=\dim(X_v)+1$. Equivalently,
$$
v\leq^1 w \iff \exists \, \beta\in\Phi^+    : \qquad v=s_\beta w \, \mbox{ and } \, \ell(v)=\ell(w)-1.
$$

We denote by $\leq_L$ the left weak Bruhat order, which can be defined by $v\leq_L w$ if and only if $\Phi(v)\subset
\Phi(w)$.
It is  generated by the covering relations:
$v\leq_L^1 w$.
This condition means that $\Phi(v)\subset \Phi(w)$ and
$\sharp\Phi(w)=\sharp\Phi(v)+1$.
Equivalently, 
$$
v\leq^1_L w \iff \exists \,  \alpha\in\Delta : \qquad v=s_\alpha w \, \mbox{ and } \, \ell(v)=\ell(w)-1.
$$

\subsection{An open subset of the incidence variety}
\label{sec:open subset}

Fix $w_1, w_2$ and $w_3$ in $W$ satisfying \eqref{eq:hypwi}.
Consider the incidence variety $Y=Y(w_1,w_2, w_3)$ and the two
maps $\pi$ and $\eta$ defined in~\eqref{eq:defYinc}.

Set
$$
\Bcal=\{(\beta,i)\,:\, \beta\in \Phi^+,\, i\in\{1,2,3\} \mbox{ and } 
\ell(s_\beta w_i)=\ell(w_i)-1\}.
$$
Let $(\beta,i)\in\Bcal$. We set
$$
\begin{array}{c} 
  w_i'=s_\beta w_i, \qquad w_j'=w_j\quad ( j\neq i),\\
  x_0'=(w_1'\inv \uo , w_2'\inv \uo , w_3'\inv \uo) \in (G/B)^3.
\end{array}
$$
Notice that these elements depend on the pair $(\beta, i)$ rather than on $i$. Moreover, we denote
$$ 
\begin{array}{cl}
  D_{(\beta,i)}&=Y(w_1',w_2',w_3')\\
  &=\{p=(z,g_1\uo,g_2\uo,g_3\uo)\in (G/B)^4\, : \, z \in g_1X_{w_1'}\cap g_2X_{w_2'}\cap g_3X_{w_3'} \},\\
  &=G\times_B (\overline{B^3 \cdot x_0'}),
\end{array}
$$
viewed as a subvariety of $Y(w_1,w_2,w_3)$.
Then
\begin{equation}
    \label{Bruhat divisors}
G\times_B \bar C^+=G\times_B C^+\sqcup
\bigcup_{(\beta,i)\in\Bcal 
} D_{(\beta,i)}.
\end{equation}

\subsection{\texorpdfstring{The differential of $\eta$}{The differential of eta}}

Given $\varphi\in\Phi$, denote by $\lg_\varphi$ the corresponding
weight space in the Lie algebra $\lg$ of $G$. 
For $w\in W$, set 
$$
T_w=\bigoplus_{\varphi\in\Phi(w)}\lg_{-\varphi}.
$$
The
projection $\lg \longto \lg/\lb$ gives an isomorphism between $T_w$
and the tangent space $T_{\uo}w\inv
Bw\uo$ of $w\inv X_w^\circ$ at the point $\uo$. From now on we identify these two spaces.

\begin{lemma}
  \label{lem:Teta=interT}
  Let $(\beta,i)\in\Bcal$ and $g_1, g_2, g_3$ in $B$. Then
  \begin{enumerate}
  \item $\Ker (T_{[e:x_0]}\eta)\simeq T_{w_1}\cap T_{w_2}\cap T_{w_3}=\{0\}$;
    \item $\Ker (T_{[e:(g_1,g_2,g_3)x_0']}\eta_{|D_{(\beta,i)}})\simeq g_1T_{w_1'}\cap
      g_2T_{w_2'}\cap g_3T_{w_3'}$.
     \item $\Ker (T_{[e:x_0']}\eta)=\{0\}$ if $\beta$ is a simple root.
  \end{enumerate}
\end{lemma}

\begin{proof}
The equality $T_{w_1}\cap T_{w_2}\cap T_{w_3}=\{0\}$ is a direct consequence of Assumption~\eqref{eq:hypwi}.
The rest of the two first assertions is direct translation of Proposition~\ref{prop:Kereta}.
For the last assertion, assume without lost of generality that $i=1$. 
Then, Proposition~\ref{prop:Kereta} asserts that 
$\Ker (T_{[e:x_0']}\eta)= T_\uo s_\beta X_{w_1}\cap T_\uo X_{w_2}\cap T_\uo X_{w_3}$. 
But, the fact that $\beta$ is simple implies that $X_{w_1}$ is stable by the minimal parabolic subgroup associated to $\beta$. 
Hence, $s_\beta X_{w_1}=X_{w_1}$.
\end{proof}

Recall that $\eta^\circ$ denotes the restriction of $\eta$ to the open subset $Y^\circ$ of $Y.$
\begin{prop}
  \label{prop:rameta}
  \begin{enumerate}
    \item The map $\eta^\circ$ is smooth.
    In particular, $\Omega=\eta(G\times_B C^+)$ is open in $X$.
    \item If $\eta$ is birational then $\eta^\circ : G \times_B C^+ \longto \Omega$ is an isomorphism.
    \item For any $(\alpha,i)\in \Bcal$ with $\alpha$ simple, the divisor $D_{(\alpha,i)}$ is not contained in $\Supp(R_\eta)$.
  \end{enumerate}
\end{prop}

\begin{proof}
  Since $G\times_B  C^+$ and $X$ are smooth, it remains to prove that
  $T_p\eta$ is invertible for any $p\in G\times_B  C^+$.
  Consider the set $Z$ of points $p$ of $G \times _B  C^+$ such that $T_p\eta$ is not
  invertible. 

  Our assumption on $(w_1,w_2,w_3)$ and Lemma~\ref{lem:Teta=interT} imply that  $T_{[e:x_0]}\eta$ is invertible.
  
  Let $\tau$ be a dominant regular one parameter subgroup of $T$.
  It is well known that for any $z\in C^+$,
  $\lim_{t\to 0}\tau(t)z=x_0$.
  Since $Z$ is closed and stable by the action of $\tau$, this implies
  that $Z\cap C^+=\emptyset$ (here $C^+$ identified to a subvariety of
  $G\times_B C^+$ by the map $x\mapsto [e:x]$).  

  As the map $\eta$ is $G$-equivariant, we deduce that $Z$ is empty.
  The fist assertion follows.

  \medskip
  Recall that  $Y^\circ=  G \times_B C^+$. 
  Fix $q \in \Omega$ and denote by $Y^\circ_q$ its schematic fiber for $\eta^\circ$. 
  Since $\eta^\circ$ is smooth of relative dimension zero and $Y^\circ$ is of finite type, we have that
  $Y^\circ_q$ is a variety of dimension zero, hence affine. 
  Moreover, by flatness of $\eta^\circ$, $\dim \CC[Y^\circ_q]=\deg \eta^\circ =1$. 
  Hence, $Y^\circ_q$ is a single point. The second assertion follows from the Zariski main theorem.

  \medskip
  The last statement follows immediately from the last assertion of Lemma~\ref{lem:Teta=interT}.
\end{proof}

\begin{remark}
  Proposition~\ref{prop:rameta} shows that 
$$
\Supp(R_\eta)\subset\bigcup_{\begin{array}{c}
  (\beta,i)\in\Bcal\\
  \beta\not\in\Delta
\end{array}}
D_{(\beta,i)}.
$$
In the next section we prove that the $D_{(\beta,i)}$ appearing on the right hand side of the previous expression are contracted by $\eta$.
  This proves that the previous inclusion is an equality. 
\end{remark}

\subsection{The case of Poincaré duality}

Let $w\in W$. Keep the notation of the previous section assuming
in addition that $(w_1,w_2,w_3)=(w,w^\vee,w_0)$.
Then, by Poincaré duality, $\eta$ is birational.
We now describe the
behaviour of the divisors on the boundary of $Y^\circ$ using Proposition~\ref{prop:Zariskidivisor}.

\begin{prop}\label{prop:divPoincare}
  Let $(\beta,i)\in \Bcal$.
  Then
  \begin{enumerate}
  \item If $\beta\in\Delta$ then 
    the restriction of $\eta$ to $D_{(\beta,i)}$ is birational.
  Moreover, there are exactly $2\rk(G)$ such divisors.
    \item If $\beta\not \in\Delta$ then 
      $D_{(\beta,i)}$ is contracted.
  \end{enumerate}
\end{prop}

\begin{proof}
  Assume that $\beta=\alpha\in\Delta$. By Proposition \ref{prop:rameta}, $D_{(\alpha,i)}$ is not
  contained in the ramification divisor. Hence,
  Proposition~\ref{prop:Zariskidivisor} implies that the restriction
  of $\eta$ to $D_{(\alpha,i)}$ is birational.

  \medskip
  Observe that the set $\{(\alpha,i)\in\Bcal\,:\,\alpha\in\Delta\}$ is
  equal to
  $$
\{(\alpha,1)\,:\,\alpha\in
D(w)\}\sqcup \{(\alpha,2)\,:\,\alpha\in D(w^\vee\}\sqcup\{(\alpha,3)\,:\,\alpha\in\Delta\},
$$
where we recall that $D(w)$ denotes the set of left descents of $w$. But notice that
$$
\forall\alpha\in\Delta, \qquad \alpha\not\in D(w)\iff -w_0\alpha\in D(w^\vee).
$$
Thus, this set has cardinality $2\rk(G)$. This proves statement $(i).$

\medskip
\noindent\underline{Claim.} The closed subset $X-\Omega$ has $2\rk(G)$
irreducible components of codimesion one in $X$.

\medskip

Note that $G \times_B C^+$ contains $U^- \times C^+$ as an open
subset. 
The latter being an affine space, it follows that $\CC[G \times_B
C^+]^*= \CC^*$.
By Proposition~\ref{prop:rameta}, $\CC[\Omega]^*=\CC^*.$
Let $E_1,\dots,E_s$  be the irreducible
components of $X-\Omega$ of codimension one in $X$. The previous discussion and the fact that $X$ is smooth 
imply that we have an exact sequence (see e.g. \cite[Proposition 6.5]{Hart}):
\begin{equation}
0\longto \oplus_{j=1}^s\ZZ E_s\longto
\Pic(X)\longto\Pic(\Omega)\longto 0.\label{eq:Picexact}
\end{equation}
Observe first that $\Pic(X)$ is a free abelian group of
rank $3\rk(G)$. 
Then, the irreducible components of the complement of $U^- \times C^+$ in $G \times_B C^+ $ are the
pullbacks by $\pi$ of the divisors $\overline{B^-s_\alpha \uo}$. There are
$\rk(G)$ of them. 
Using the exact sequence analogue to \eqref{eq:Picexact} and Proposition~\ref{prop:rameta} we 
deduce that $\Pic(\Omega)\simeq\Pic(G\times_B C^+)$ is a free abelian group of rank $\rk(G)$.

Now, the exactness of sequence \eqref{eq:Picexact} implies that
$s=2\rk(G)$, proving the claim.

\medskip
Since $\eta$ is proper and dominant, it is surjective.
Then, Proposition~\ref{prop:Zariskidivisor} implies that for any
$j=1,\dots,2\rk(G)$, there exists a unique irreducible component $D_j$
of $Y-Y^\circ$ such that $\eta(D_j)=E_j$.
Thus, these $2\rk(G)$ divisors $D_j$'s are exactly the $D_{(\alpha,i)}$ for
$(\alpha,i)\in\Bcal$ with $\alpha\in \Delta$.
In other words, the other irreducible components of $Y-Y^\circ$,
namely the $D_{(\beta,i)}$ for
$(\beta,i)\in\Bcal$ with $\beta\not\in \Delta$, are contracted by $\eta$.
\end{proof}

\section{The kernel of the differential map}

Fix $(w_1,w_2,w_3)\in W^3$ satisfying Assumption~\eqref{eq:hypwi} and
consider the map $\eta\,:\,Y\longto X$ defined in
Section~\ref{sec:incidence}.

\begin{prop}
  \label{prop:Dcontract}
Let $(\beta,i)\in\Bcal$ such that $ \beta\not\in\Delta$.
Then, $D_{(\beta,i)}$ is contracted by $\eta$.
\end{prop}

The main aim of this section is to prove Proposition~\ref{prop:Dcontract}.  
Observe first that it is sufficient to prove the following assertion.

\medskip
\noindent
\underline{Claim:}
For any $x\in D=D_{(\beta,i)}$, the linear map $T_x\eta_{|D}$ is not injective.

\medskip
Indeed, since we work over $\CC$, we have that $\eta$ is separated. Moreover, by
$G$-equivariance and semicontinuity of the dimension of the kernel of the differential of a morphism, 
it is sufficient to prove the claim for $x\in(\{1\}\times U^2)\cdot x_0'$. 
Recall the notation of Section \ref{sec:open subset}.

\subsection{A description as the kernel of a matrix}
\label{sec:exple}

From now on, an element $(\beta, i) \in \Bcal$ as in Proposition \ref{prop:Dcontract} is fixed. Up to $S_3$-symmetry, assume that $i=1$.
It is convenient to set
$$
w=w_1\quad x=w_2^\vee \quad y=w_3^\vee\quad v=s_\beta w_1 \quad D=D_{(\beta,i)}.
$$
Then, Assumption~\eqref{eq:hypwi} is equivalent to
\begin{equation}
  \label{eq:4}
  \Phi(w)=\Phi(x)\sqcup\Phi(y),
  \quad
  v\leq^1 w
\quad
 \mbox{and} \quad
  \Phi(v)\not\subset\Phi(w).
\end{equation}
For any $\varphi\in \Phi$, fix  nonzero elements $\xi_\varphi$ and
$\xi^\varphi$ in $\lg_{-\varphi}$ and $(\lg^*)_{\varphi}$ respectively.

Fix $g_x$ and $g_y$ in $U$.
To this data, we attach a matrix $M=M(v,w,x,y,g_x,g_y)$ whose  rows are indexed by
$\Phi(w)$ and  columns by $\Phi(v)$. The entry at row
$\beta\in\Phi(w)$ and column  $\gamma\in\Phi(v)$ is
$$
M_{\beta\gamma}=\left\{
  \begin{array}{ll}
    \xi^\beta(g_x\inv\xi_\gamma)&{\rm if\ } \beta\in\Phi(x)\\
    \xi^\beta(g_y\inv\xi_\gamma)&{\rm if\ } \beta\in\Phi(y)
  \end{array}
  \right .
$$

\begin{lemma}
  \label{lem:TKerM}
  The Kernel of $M$ is isomorphic to the intersection
  $$
T_v\cap g_x T_{x^\vee}\cap g_y T_{y^\vee}\simeq T_{p}\eta_{|D},
  $$
where $p=[e:(\uo,g_xx\inv \uo,g_y y\inv \uo)]$.
\end{lemma}

\begin{proof}
The isomorphism  $T_v\cap g_x T_{x^\vee}\cap g_y T_{y^\vee}\simeq T_{p}\eta_{|D}$ was proved in Lemma~\ref{lem:Teta=interT}. 
The fact that the intersection is the kernel of $M$ is obvious. 
\end{proof}

We deduce

\begin{lemma}
\label{lem:kerM non zero}
  If the Kernel of $M=M(v,w,x,y,g_x,g_y)$ is nonzero for any $g_x,g_y\in U^2$ then $D=D_{(\beta,i)}$ is contracted.
\end{lemma}

\begin{exple}
In the root system $D_4$, consider $w=s_2s_3s_1s_2s_4s_2$ and
$v=ws_2$. Then
$$
\Phi(w)=\{\alpha_{2}, \alpha_{2} + \alpha_{4}, \alpha_{4}, \alpha_{1}
+ \alpha_{2} + \alpha_{4}, \alpha_{2} + \alpha_{3} + \alpha_{4},
\alpha_{1} + 2\alpha_{2} + \alpha_{3} + \alpha_{4}\}$$
and
$$
\Phi(v)=\{
\alpha_{4}, \alpha_{2} + \alpha_{4}, \alpha_{1} + \alpha_{2} + \alpha_{4}, \alpha_{2} + \alpha_{3} + \alpha_{4}, \alpha_{1} + \alpha_{2} + \alpha_{3} + \alpha_{4}
\}.
$$
Consider a generic matrix
$$
\xi=\left(\begin{array}{rrrr|rrrr}
0 & x_{0} & x_{1} & x_{2} & x_{6} & x_{7} & x_{8} & 0 \\
0 & 0 & x_{3} & x_{4} & x_{9} & x_{10} & 0 & - x_{8} \\
0 & 0 & 0 & x_{5} & x_{11} & 0 & - x_{10} & - x_{7} \\
0 & 0 & 0 & 0 & 0 & - x_{11} & - x_{9} & - x_{6} \\
\hline
 0 & 0 & 0 & 0 & 0 & - x_{5} & - x_{4} & - x_{2} \\
0 & 0 & 0 & 0 & 0 & 0 & - x_{3} & - x_{1} \\
0 & 0 & 0 & 0 & 0 & 0 & 0 & - x_{0} \\
0 & 0 & 0 & 0 & 0 & 0 & 0 & 0
\end{array}\right)
$$
in Lie$(U)$ and set $u=\exp(\xi)$. The matrix $M(v,w,w,e,u,e)$ is

\bigskip
\noindent
\resizebox{1\hsize}{!}{
$
\left(\begin{array}{c|ccccc}
  &\alpha_4&\alpha_2+\alpha_4&\alpha_1+\alpha_2+\alpha_4&\alpha_2+\alpha_3+\alpha_4&\alpha_1+\alpha_2+\alpha_3+\alpha_4\\
        \hline
\alpha_4&1 & - x_{3} & \frac{1}{2} x_{0} x_{3} -  x_{1} & \frac{1}{2}
                                                          x_{3} x_{5}
                                                          + x_{4} &
                                                                    -\frac{1}{3} x_{0} x_{3} x_{5} - \frac{1}{2} x_{0} x_{4} + \frac{1}{2} x_{1} x_{5} + x_{2} \\
\alpha_2&0 & - x_{11} & x_{0} x_{11} & x_{5} x_{11} & - x_{0} x_{5} x_{11} \\
\alpha_2+\alpha_4&0 & 1 & - x_{0} & - x_{5} & x_{0} x_{5} \\
\alpha_1+\alpha_2+\alpha_4&0 & 0 & 1 & 0 & - x_{5} \\
\alpha_2+\alpha_3+\alpha_4&0 & 0 & 0 & 1 & - x_{0} \\
\alpha_1 +2\alpha_2 +\alpha_3 +\alpha_4&0 & 0 & 0 & 0 & 0
\end{array}\right)
$
}

\bigskip
Its kernel is nontrivial since two rows are proportional. The authors
did not find such a general reason to prove that the kernel of any $M$
is nontrivial. Instead, we use the fact that the result is known in the
case of Poincaré duality, and we reduce to it by means of Theorem \ref{th:combi}.    
\end{exple}

\subsection{The case when a submatrix is strictly triangular}

Define the height $h\,:\,\Phi\longto\ZZ$, by $h(\sum_{\alpha\in\Delta}n_\alpha
\alpha)=\sum_{\alpha\in\Delta}n_\alpha$.
For $h\in\ZZ$, we set
$\Phi(w)_h=\{\varphi\in\Phi(w)\,:\,h(\varphi)=h\}$ and $\Phi(w)_{\leq
  h}=\{\varphi\in\Phi(w)\,:\,h(\varphi)\leq h\}$. Note that if $\varphi \leq \psi$, then $h(\varphi)\leq h(\psi)$.

The following lemma is well-known:

\begin{lemma}
  \label{lem:actionU}
  Let $\gamma\in\Phi$.
  For any $g\in U$ and $\xi\in\lg_{-\gamma}$, we have
  $$
g\xi\in\xi+\sum_{\psi<\gamma}\lg_{-\psi}.
  $$
\end{lemma}

As an immediate consequence of Lemma~\ref{lem:actionU}, we get that 
$$
\begin{array}{lcl}
  M_{\beta\gamma}=1&{\rm if}& \beta=\gamma,\\
   M_{\beta\gamma}\neq 0&{\rm implies}& \beta\leq\gamma.
\end{array}
$$
We improve this fact as follows.

\begin{lemma}
  \label{lem:forme}
  If at least one of the following assertions holds:
  \begin{enumerate}
  \item $\exists h\in\NN^*$ such that $\sharp\Phi(v)_{\leq
      h}>\sharp\Phi(w)_{\leq h}$;
    \item $\exists h\in\NN^*$ such that $\sharp\Phi(v)_{\leq
      h}=\sharp\Phi(w)_{\leq h}$  and $\Phi(v)_{h+1}\not\subset\Phi(w)$;
  \end{enumerate}
  then $\Ker M\neq\{0\}$.
\end{lemma}

\begin{proof}
  First, number the elements of $\Phi(w)$ (and independently $\Phi(v)$)
  in such a way that the map $\beta\mapsto h(\beta)$ is nondecreasing.
  Let $N$ be the submatrix of $M$ with rows and columns in $\Phi(w)_{\leq
    h}$ and $\Phi(v)_{\leq h}$ respectively.
  Lemma~\ref{lem:actionU} implies that $M$ has the following form
  
  \begin{equation}
  \begin{pmatrix}
    N&\star\\
    0&\star
  \end{pmatrix}\label{eq:formM1}
\end{equation}
 
  With the first assumption of the lemma, $N$ has more columns than rows;
  hence its kernel is not reduced to zero. By \eqref{eq:formM1}, that of
  $M$  too.
  
  \medskip

 Assume now that we are in the second case. 
Then, $N$ is a square matrix
 and we can fix $\gamma\in\Phi(v)_{h+1}$
 such that $\gamma\not\in\Phi(w)$.
 Up to renumbering, assume that $\gamma$ is the first root in
 $\Phi(v)_{h+1}$.

 Let $\tilde N$ be the submatrix of $M$ with rows and columns in $\Phi(w)_{\leq
    h}$ and $\Phi(v)_{\leq h}\cup\{\gamma\}$ respectively.
 Lemma~\ref{lem:actionU} implies that $M$ is block triangular as in
 \eqref{eq:formM1}
 with $\tilde N$ in place of $N$.
 Hence, the kernel $M$ is not reduced to zero. 
\end{proof}

\subsection{The case when no submatrix is strictly triangular}

We now assume that Lemma~\ref{lem:forme} does not apply; that is
that:
\begin{enumerate}
\item[(H1)]
  $\forall h\in\NN^*\qquad \sharp\Phi(v)_{\leq
      h}\leq \sharp\Phi(w)_{\leq h}$; and
\item[(H2)] $\forall h\in\NN^*$ such that $\sharp\Phi(v)_{\leq
      h}=\sharp\Phi(w)_{\leq h}$ we have $\Phi(v)_{h+1}\subset\Phi(w)$.
  \end{enumerate}

  Observe that (H2) can be re-written as
  \begin{center}
    (H2') $\forall h\in\NN^*\qquad \sharp\Phi(v)_{<
      h}=\sharp\Phi(w)_{< h}\implies \Phi(v)_{h}\subset\Phi(w)$.
  \end{center}
Set
$$
\begin{array}{l}
\Phi(w)-\Phi(v)=\{\beta_0,\beta_1,\dots,\beta_s\}\\
\Phi(v)-\Phi(w)=\{\gamma_0,\gamma_1,\dots,\gamma_t\}
\end{array}
$$
by labeling the elements by nondecreasing height.

A key result to understand the matrix $M$ is the following

\begin{prop}
\label{prop:decoupage M}
  With above notation and assuming (H1) and (H2), we have
  \begin{enumerate}
\item $w=vs_{\beta_0}$;
  \item $s=t+1$;
\item $h(\beta_0)<h(\gamma_0)<h(\beta_1)<h(\gamma_1)<\cdots<h(\beta_s)$;
\item for any $i=0,\dots,t$, there exists $k_i\in\NN^*$ such that $\beta_{i+1}=\gamma_i+k_i\beta_0$.
  \end{enumerate}
\end{prop}

\begin{proof}
  Set $\delta=v\inv \beta$ in such a way that $w=vs_\delta$.  By an immediate induction, it is sufficient to prove the following
  three assertions:
  \begin{enumerate}
  \item[(P0)] $\Phi(w)_{\leq h(\delta)}=\Phi(v)_{\leq
      h(\delta)}\sqcup\{\delta\}$, and $\beta_0=\delta$.
\item[(P1)] If 
  \begin{enumerate}
  \item $h(\beta_0)<h(\gamma_0)<\cdots<h(\beta_i)$,
\item $\Phi(w)_{\leq h(\beta_i)}-\Phi(v)=\{\beta_0,\beta_1,\dots,\beta_i\}$,
\item $\Phi(v) _{\leq
    h(\beta_i)}-\Phi(w)=\{\gamma_0,\gamma_1,\dots,\gamma_{i-1}\}$, 
\item \label{hyp2} $\forall \, \, 0\leq j<i, \qquad\exists \, \,  k\in\NN : \qquad
  \beta_{j+1}-k\beta_0=\gamma_j$, and
\item\label{hyp1}  $\Phi(w)_{>h(\beta_i)}\neq \Phi(v)_{>h(\beta_i)}$
  \end{enumerate}
then
\begin{enumerate}
\setcounter{enumii}{5}
\item $h(\gamma_i)>h(\beta_i)$,
\item $\Phi(w)_{\leq
    h(\gamma_i)}-\Phi(v)=\{\beta_0,\beta_1,\dots,\beta_i\}$, and
\item $\Phi(v) _{\leq h(\gamma_i)}-\Phi(w)=\{\gamma_0,\gamma_1,\dots,\gamma_{i}\}$. 
\end{enumerate}
\item[(P2)] If 
  \begin{enumerate}
  \item $h(\beta_0)<h(\gamma_0)<\cdots<h(\gamma_i)$,
\item $\Phi(w)_{\leq
    h(\gamma_i)}-\Phi(v)=\{\beta_0,\beta_1,\dots,\beta_i\}$, 
\item $\Phi(v) _{\leq
    h(\gamma_i)}-\Phi(w)=\{\gamma_0,\gamma_1,\dots,\gamma_{i}\}$, and 
\item $\forall \, \,  0\leq j<i, \qquad\exists \, \,  k\in\NN : \qquad \beta_{j+1}-k\beta_0=\gamma_j$
  \end{enumerate}
then
\begin{enumerate}
\setcounter{enumii}{4}
\item $h(\beta_{i+1})>h(\gamma_i)$,
\item $\Phi(w)_{\leq h(\beta_{i+1})}-\Phi(v)=\{\beta_0,\beta_1,\dots,\beta_{i+1}\}$,
\item $\Phi(v) _{\leq
    h(\beta_{i+1})}-\Phi(w)=\{\gamma_0,\gamma_1,\dots,\gamma_{i}\}$, and
\item $\exists \, \,  k\in\NN: \qquad \beta_{i+1}-k\beta_0=\gamma_i$.
\end{enumerate}
  \end{enumerate}

\noindent {\sc Proof of }(P0).
Fix a positive root $\theta \neq \beta_0$.
It is well known that there exist integers  $p\leq q$ such that
$$ 
(\theta+\ZZ\beta_0)\cap \Phi^+=\{\theta+k\beta_0\,:\, k\in
[p;q]\cap\ZZ\}.
$$
Since $\beta_0\in \Phi(w)$, the convexity of $\Phi(w)$ implies that 
$$ 
(\theta+\ZZ\beta_0)\cap \Phi(w)=\{\theta+k\beta_0\,:\, k\in
[r;q]\cap\ZZ\},
$$
for some integer $p\leq r\leq q+1$. 
Then 
\begin{equation}
  \label{eq:Phiwbeta}
  (\theta+\ZZ\beta_0)\cap \Phi(v)=\{\theta+k\beta_0\,:\, k\in
[p;s]\cap\ZZ\},\quad\mbox{where }s=p+q-r.
\end{equation}
In other words, $(\theta+\ZZ\beta_0)\cap \Phi(w)$ consists in the 
$q-r+1$ last elements of $(\theta+\ZZ\beta_0)\cap \Phi^+$, whereas $(\theta+\ZZ\beta_0)\cap \Phi(v)$ consists in the
$q-r+1$ first elements. In \cite{distrib}, there is a geometric proof
of equality~\eqref{eq:Phiwbeta}. 
For completeness, we include a combinatorial one. By coconvexity of $\Phi(v)$ it is sufficient to prove the following lemma. 

\begin{lemma}
    \label{lem:inversions cover}
    Let $w \in W$, $\delta \in \Phi^+$ and $v=w s_\delta \leq^1 w$. For any $\theta \in \Phi^+ \setminus\{\delta\}$ we have:
    $$ \sharp((\theta + \ZZ\delta) \cap \Phi(w))= \sharp( (\theta + \ZZ\delta) \cap \Phi(v)).$$
\end{lemma}

\begin{proof}
    Enumerate the simple roots, that is $\Delta= \{\alpha_1, \dots \alpha_r\}.$ Let $w=s_{i_m} \dots s_{i_1}$, 
    for some integers $1 \leq i_j \leq r$, be a reduced expression of $w$. 
    Then there exists a unique $1 \leq k \leq m$ such that 
    $$\delta= \begin{cases}
        s_{i_1} \dots s_{i_{k-1}} \alpha_{i_k} & \text{if} \quad 1< k\\
        \alpha_{i_1} & \text{otherwise}.
    \end{cases}  $$ 
Moreover, $v= s_{i_m} \dots s_{i_{k+1}} s_{i_{k-1}} \dots s_{i_1}.$ If
$k=m$, the statement is obvious. 
Indeed, the condition $v \leq^1_L w$ implies that  $\Phi(w)=
\Phi(v) \cup \{\delta\}$. 
Otherwise, let $w'=s_{i_m} w$ and $v'=s_{i_m}v$. We have that $v'\leq^1 w'$ and 
$$\Phi(w)= \Phi(w') \sqcup \{(w')\inv \alpha_{i_m}\} \quad \text{and} \quad \Phi(v)= \Phi(v') \sqcup \{(v')\inv \alpha_{i_m}\}.  $$
But, paying attention if $k=1$,
\begin{align*}
    (w')\inv \alpha_{i_m} = & s_{i_1} \dots s_{i_{m-1}} \alpha_{i_m}\\
    = & s_{1_1}\dots s_{i_{k-1}}s_{i_k}s_{i_{k-1}}\dots s_{i_1}(v')^\inv \alpha_{i_m}\\
    = & s_\delta (v')^\inv \alpha_{i_m} \in (v')^\inv \alpha_{i_m} + \ZZ \delta.
\end{align*}
The statement follows easily by induction on $m-k.$
\end{proof}

Going back to the proof, we have that 
$\Phi(w)_{<h(\delta)}\subset\Phi(v)$ and
$\Phi(w)_{h(\delta)}-\{\beta\}\subset\Phi(v)$. 
Using $(H1)$, we get  $\Phi(w)_{<h(\delta)}=\Phi(v)_{<h(\delta)}$.
Now, (H2) implies that $\Phi(v)_{h(\delta)}\subset \Phi(w)$. Hence, 
$\Phi(w)_{\leq h(\delta)}=\Phi(v)_{\leq
      h(\delta)}\sqcup\{\delta\}$. 

Then, $\delta$ is the unique element of minimal height in
$(\Phi(v)\cup\Phi(w))-(\Phi(v)\cap\Phi(w))$, and it belongs to
$\Phi(w)$. Hence, $\beta_0=\delta$.\\

\noindent {\sc Proof of }(P1).
Let $\theta\in (\Phi(v)\cup\Phi(w))-(\Phi(v)\cap\Phi(w))$ such that
$h(\theta)>h(\beta_i)$ and of minimal height with these
properties. Such a $\theta$ exists by Hypothesis~(P1e). 
Roughly speaking $\theta$ is the next difference. 

Consider $\theta+\ZZ\beta$. 
 By Hypothesis~(P1d), for any $0\leq j<i$, $\gamma_j\in \theta+\ZZ\beta$ if
 and only if $\beta_{j+1}\in \theta+\ZZ\beta$.
Hence, by \eqref{eq:Phiwbeta}, the root $\theta$ (that is the next difference in
$\theta+\ZZ\beta$) belongs to $\Phi(v)-\Phi(w)$.

The assumptions imply that
$\sharp\Phi(v)_{<h(\beta_i)}=\sharp\Phi(w)_{<h(\beta_i)}$. 
Hence (H2) gives $\Phi(v)_{h(\beta_i)}\subset \Phi(w)$. 
Thus, $h(\theta)\neq h(\beta_i)$ and $h(\theta)> h(\beta_i)$.

The set $\{\gamma_0,\dots,\gamma_{i-1},\theta\}\sqcup\Phi(w)_{\leq
  h(\theta)}$ is contained in $\{\beta_0,\dots,\beta_i\}\sqcup\Phi(v)_{\leq
  h(\theta)}$, and even equal by (H1). Then $\gamma_i=\theta$. 
This ends the proof of (P1). 

\medskip

\noindent {\sc Proof of }(P2).
Let $\theta\in (\Phi(v)\cup\Phi(w))-(\Phi(v)\cap\Phi(w))$ such that
$h(\theta)>h(\gamma_i)$ and of minimal height with these
properties. Such a $\theta$ exists since
$\sharp\Phi(w)=\sharp\Phi(v)+1$.

The assumptions imply that
$\sharp\Phi(v)_{<h(\theta)}=\sharp\Phi(w)_{<h(\theta)}$. 
Then (H2) gives $\Phi(v)_{h(\theta)}\subset \Phi(w)$ and
$\theta\in\Phi(w)$.

Consider $\theta+\ZZ\beta$ and recall \eqref{eq:Phiwbeta}.
For any $j<i$, $\gamma_j\in \theta+\ZZ\beta$ if
 and only if $\beta_{j+1}\in \theta+\ZZ\beta$.
We deduce that  there exists $k\in\NN$ such that $\theta-k\beta_0$
belongs to
$\Phi(v)-\{\gamma_0,\dots,\gamma_{i-1}\}$ and not in $\Phi(w)$. 
But $\gamma_i$ is the only such element of height less than
$h(\theta)$. 
Hence,  $\theta-k\beta_0=\gamma_i$.

What we have just proved also implies that
$\theta+\ZZ\beta_0=\gamma_i+\ZZ\beta_0$ and that $\theta$ is the only
element in $\Phi(w)-\Phi(v)$ of its height. It follows that
$\theta=\beta_{i+1}$ and
$\Phi(v)_{h(\theta)}\sqcup\{\theta\}=\Phi(w)_{h(\theta)}$.
This ends the proof of (P2).
\end{proof}

To emphasize the structure of $M$ in blocks, let us set, for any
$i=0,\dots,s$
$$
\begin{array}{l}
  \Phi_i^-=\{\theta\in\Phi(v)-\{\gamma_0,\dots,\gamma_{s-1}\}\;|\;h(\gamma_{i-1})\leq
  h(\theta)\leq h(\beta_i)\}\\
\Phi_i^+=\{\theta\in\Phi(v)-\{\gamma_0,\dots,\gamma_{s-1}\}\;|\;h(\beta_i)<
  h(\theta)< h(\gamma_i)\},
\end{array}
$$
where, by convention $h(\gamma_{-1})=0$ and $h(\gamma_s)=+\infty$.

For $i=0,\dots,s$, denote by $M_i^-$ the submatrix of $M$ whose 
rows and columns indices of its entries belong to $\Phi_i^-$.
For $i=0,\dots,s-1$, denote by $M_i^+$ the submatrix of $M$
corresponding to the 
rows $\Phi_i^+\sqcup\{\beta_i\}$ and columns
$\Phi_i^+\sqcup\{\gamma_i\}$.
Finally,  denote by $M_s^+$ the submatrix of $M$
corresponding to the 
rows $\Phi_s^+\sqcup\{\beta_s\}$ and columns
$\Phi_s^+$.

Observe that all the $M_i^\pm$ are square matrices excepted for $M_s^+$. 
Recall that the elements of $\Phi(v)$ and $\Phi(w)$ are numbered by
nondecreasing height. 
Then, Lemma~\ref{lem:actionU} implies easily that $M$ is
upper triangular by blocs with $M_0^-,M_0^+,M_1^-,\dots,M_s^+$ as
diagonal blocks.  
The same lemma also implies that each $M_i^-$ is upper triangular with
1's on the diagonal and that 
$$
M_s^+=
\begin{pmatrix}
  *&\cdots&*&\cdots&*\\
1&&*\\
&&\ddots&\\
&&0&&1
\end{pmatrix}.
$$

On the example in Section~\ref{sec:exple}, $s=1$, $M_0^-$ is the
identity matrix of size 1,  $M_0^+$ is the $(4\times 4)$-submatrix with rows
in $\{2,3,4,5\}$ and columns in $\{2,3,4,5\}$. The matrix $M_0^+$ is
empty in this case since $\Phi_1^+$ is.

\medskip

Then, elementary linear algebra gives

\begin{lemma}
  \label{lem:KerMi}
With above notation, we have
$\Ker M\neq\{0\}$ if and only if there exists $i\in\{0,\dots,s-1\}$
such that $M_i^+$ is not invertible. 
\end{lemma}
 

\subsection{The trick using Poincaré duality}

Recall that the matrix $M$ depends on the choice of a pair $(g_x,g_y)$
of elements in the unipotent group $U$, nevertheless the sets
$\Phi_i^\pm$, and the existence of a corresponding block subdivision
of $M$, only depend on $v$ and $w$.  

\begin{lemma}
  \label{lem:kerMH1H2}
  Under the assumptions (H1) and (H2), the Kernel of $M=M(v,w,x,y,g_x,g_y)$ is nonzero, for any $g_x,g_y\in U$.
\end{lemma}

\begin{proof}
First consider the case of Poincaré duality: $M(v,w,w,e,g,g')$ for $g$,
$g'$ in $U$. Since $\Phi(e)=\emptyset$, the matrix is independent of $g'$.

By Proposition~\ref{prop:divPoincare}, the divisor associated to
$(v,w^\vee,w_0)$ is contacted by $\eta$.
Then the Kernel of $M(v,w,w,e,g,g')$ is nonzero for any $g\in U$.
By Lemma~\ref{lem:KerMi} this implies that $\prod_i\det
M_i^+(v,w,w,e,g,g')=0$ for any $g\in U$.
Since $U$ is an irreducible  variety, this implies that there exists
$i_0$ such that $\det M_{i_0}^+(v,w,w,e,g,g')=0$ for any $g,g'\in U$.

\medskip

Consider now the matrix $M(v,w,x,y,g_x,g_y)$.
By Proposition~\ref{prop:calculdet} of Section~\ref{sec: A determinant}, the  only coefficients occurring
in 
$$
\det M_{i_0}^+(v,w,x,y,g_x,g_y)
$$
are indexed by roots  in the
interval $[\beta_{i_0};\gamma_{i_0}]$. By Proposition~\ref{prop:decoupage M} there exists a nonnegative 
integer $k$ such that $\gamma_{i_0} + k \beta_0 \not \in \Phi(w)$ and $\gamma_{i_0} + (k+1) \beta_0 \in \Phi(w).$ 
Moreover, $[\beta_{i_0}, \gamma_{i_0}] \subseteq [\beta_0 ; \gamma_{i_0} + k \beta_0]$ by Statement 4 of 
Proposition~\ref{prop:decoupage M}.

Then, we can apply Theorem~\ref{th:combi} with $\beta=\beta_0$ and $\gamma=\gamma_{i_0}+k\beta_0$, 
and we deduce that 
$$
\det M_{i_0}^+(v,w,x,y,g_x,g_y)
$$
is equal to either 
$\det M_{i_0}^+(v,w,w,e,g_x,g_x)$ or $\det M_{i_0}^+(v,w,e,w,g_y,g_y)$.
By the first part of the proof, $\det M_{i_0}^+(v,w,w,e,g,g')=0$ for any $g,g'\in U$. 

But, $\det M_{i_0}^+(v,w,e,w,g',g)=\det M_{i_0}^+(v,w,w,e,g,g')$.
Thus, 
$$
\det M_{i_0}^+(v,w,x,y,g_x,g_y)=0.
$$ 
With Lemma~\ref{lem:KerMi}, this concludes. 
\end{proof}

\subsection{Proof of Theorem~\ref{th:main}}

Lemmas~\ref{lem:kerMH1H2}, \ref{lem:forme} and \ref{lem:kerM non zero} implies the $D_{(\beta,i)}$ is contracted 
for any $(\beta,i)\in\Bcal$ with $\beta\not\in\Delta$. 
Now, Proposition~\ref{prop:rameta} implies that the ramification divisor of $\eta$ is contracted.  
Thus, Proposition~\ref{prop:c=1contract} implies that $\eta$ is birational. 
Finally, Proposition~\ref{prop:eta bir} allows to conclude.

\subsection{Proof of Theorem~\ref{th:combi}}
\label{sec:proof thm combi}

First, notice that the statement of Theorem~\ref{th:combi} only depends on the root system $\Phi$ and its set of positive roots $\Phi^+$. 
That is, the group $G$ plays no role.

We start with some simple observations.

\begin{lemma}
  \label{lem:thetax}
  In the setting of Theorem~\ref{th:combi}, we have that $\gamma+\beta\in\Phi_1$.
\end{lemma}
\begin{proof}
If   $\gamma+\beta\not\in\Phi_1$ then $\gamma+\beta\in\Phi_2$. But
$\beta\not\in \Phi_2$ and $\gamma \not\in \Phi_2$. Contradiction
with the coconvexity of $\Phi_2$.
\end{proof} 

\begin{lemma}
\label{lem:reducible root syst}
Suppose that the root system $\Phi$ is reducible, and let $\Phi=\Phi^1 \sqcup \Phi^2$ be an orthogonal decomposition of $\Phi$, for some root systems $\Phi^1$ and $\Phi^2$. 
If  Theorem~\ref{th:combi} holds for $\Phi^1$ and $\Phi^2$, then it holds for $\Phi$.
\end{lemma}

\begin{proof}
 Let $\delta \in \Phi^1$. 
 The key observation is that, for a root $\varphi \in \Phi$, we have that $\varphi \leq \delta$ (resp. $\delta \leq \varphi$) in the root system $\Phi$ if and only if $\varphi \in \Phi^1$ and $\varphi\leq \delta$ (resp. $\delta \leq \varphi$) in $\Phi^1$.

 Let $\Phi_1,\Phi_2,\Phi_3, \beta$ and $\gamma$ as in the statement of Theorem~\ref{th:combi}.
Without loss of generality, we have that $\beta \in \Phi^1$. 
Assume that $\beta < \gamma$, since otherwise there is nothing to prove.
Then, we have that $\gamma \in \Phi^1$ and $[\beta, \gamma] \subseteq \Phi^1.$
Noticing that the sets $\Phi_j^1:=\Phi^1\cap \Phi_j$ (for
$j \in \{1,2,3\}$) are biconvex,
the fact that $[\beta, \gamma] \cap \Phi_2$ is empty follows from  the aforementioned key observation and the fact that Theorem~\ref{th:combi} holds for the root system $\Phi^1$.
\end{proof}

For a positive root $\delta= \sum_{\alpha \in \Delta}n_\alpha \alpha$ of $\Phi$, we define the support of $\delta$ as
$$
\supp(\delta):= \{\alpha \in \Delta \, : \, n_\alpha \neq 0\}.
$$

From now on, for any subset $Y$ of $\Phi^+$ we denote by $Y^c:= \Phi^+ \setminus Y.$ 

\begin{lemma}
\label{lem:opssupp}
It is sufficient to prove Theorem~\ref{th:combi} when $\supp(\gamma)=\Delta$.
\end{lemma}

\begin{proof}
By Lemma~\ref{lem:reducible root syst}, we may assume that $\Phi$ is irreducible.
We argue by induction on the cardinality of $\Delta$. 
Note that if $|\Delta|=1$, Theorem~\ref{th:combi} is obvious.
By the inductive hypothesis and Lemma~\ref{lem:reducible root syst} we can assume that $\Phi$ is irreducible. 
Assume that $\supp(\gamma)$ is strictly contained in $\Delta$, and let $F$ be the real span of $\supp(\gamma)$ in $\RR \Phi$.
Observe that, for $i=1,2,3$,  the set $\Phi_{i,F}:=\Phi_i\cap F$ is biconvex in $F\cap \Phi$.
  Moreover, we have that  $\Phi_{3,F}=\Phi_{1,F}\cup\Phi_{2,F}$ and $\Phi_{1,F}\cap\Phi_{2,F}= \emptyset$.
  Note that a root $\varphi \in \Phi $ satisfies $\beta < \varphi < \gamma$ in the root system $\Phi$ if and only if 
  $\beta, \varphi \in \Phi \cap F$ and the condition $\beta < \varphi < \gamma$ holds in $\Phi \cap F$. 
 Since the cardinality of the set of simple roots of $\Phi \cap F$ is strictly smaller than $|\Delta|$, 
 the inductive hypothesis allows to conclude.
\end{proof}

From now on, we assume that
$$
\Phi\mbox{ is irreducible and }\supp(\gamma)=\Delta.
$$ 
Proving Theorem \ref{th:combi} under this assumption is sufficient because of Lemma~\ref{lem:reducible root syst} and \ref{lem:opssupp}.
In order to do so, we need to introduce some notation and recall some recent advancements on root systems and their quotients from \cite{Dim:quotroot}.

\medskip

Let $I$ be a possibly empty subset of $\Delta$, and $\ZZ I$ be the integer span of $I$ inside the real span of $\Phi.$
We denote by $\Phi_I$ the root system $\Phi \cap \ZZ I$, considered with its set of positive roots $\Phi_I^+:= \Phi^+ \cap \ZZ I.$
Let $\pi_I: \Phi \longto \ZZ \Phi / \ZZ I$ be the natural map, and consider the sets 
$$
\Phi/I:= \pi_I \bigl( \Phi \setminus \Phi_I), \qquad (\Phi/I)^+:= \pi_I \bigl( \Phi^+ \setminus \Phi_I^+).
$$
For two possibly empty sets $\Psi \subseteq (\Phi/I)+$ and $X \subseteq \Phi_I^+$, the set 
$$
\Inf_I^\Delta( \Psi, X):= \pi_I^{-1}(\Psi) \cup X
$$
is the \textit{inflation} from $I$ to $\Delta$ of $\Psi$ by $X$. 
The notion of inflation has been defined in \cite{Dim:quotroot}.
A subset $Y$ of $\Phi^+$ is said to be \textit{inflated} from $I$ is there exists  $\Psi \subseteq (\Phi/I)+$ and $X \subseteq \Phi_I^+$ such that 
$$
Y= \Inf_I^\Delta( \Psi, X).
$$
Notice that in that case the sets $\Psi$ and $X$ are unique.
If moreover the sets $I$, $\Psi$ and $X$ satisfy the conditions of \cite[Theorem 3.14]{Dim:quotroot}, we say that $Y$ is \textit{canonically inflated} from $I$.
By \cite[Theorem 3.14]{Dim:quotroot}, any subset $Y$ of $\Phi^+$ is canonically inflated from a unique subset $I$ of $\Delta$. 
Moreover, the set $I$ is a proper subset of $\Delta$. 
In other words, $I$ is strictly contained in $\Delta$.

\begin{lemma}
    \label{lem: inflation}
  There exists a proper subset $I$ of $\Delta$ satisfying the following conditions.
    \begin{enumerate}
        \item $\Phi_2$ is contained in $\Phi_I^+$.
        \item $\pi_I\bigl( \Phi_1 \setminus \Phi_I \bigr) \cap \pi_I\bigl( \Phi_3^c \setminus \Phi_I \bigr) = \emptyset.$
    \end{enumerate}
\end{lemma}

\begin{proof}
   Let $\Phi_4:= \Phi_3^c$. 
Let $i \in \{1,2,4\}$ be such that the longest root of $\Phi$ belongs to $\Phi_i$.
Moreover, let $I$ be the subset of $\Delta$ from which $\Phi_i$ is canonically inflated.
We claim that the previously defined set $I$ satisfies the requirements of the lemma.

It is clear that $I$ is a proper subset of $\Delta$ \cite[Theorem 3.14]{Dim:quotroot}.
   By \cite[Theorem 5.9]{Dim:quotroot}, there exist $j \in \{1,2,4\}\setminus \{i\}$ such that 
   $$
   \Phi_j= \Inf_I^\Delta( \emptyset, X)
   $$
   for some subset $X $ of $ \Phi_I^+$.
   By the definition of inflation, it follows that $\Phi_j=X \subseteq \Phi_I^+$.
   Since $\supp(\gamma)=\Delta= \supp(\gamma +\beta)$, neither $\gamma$ nor $\gamma + \beta$ are contained in $\Phi_I^+$. 
   It follows that $j=2$.
   Hence, the first condition of the lemma is fulfilled for the chosen set $I$.
   Then, note that if $Y$ is a subset of $\Phi^+$ which is inflated from $I$, we have that for any  $x \in (\Phi/I)^+$ either $\pi_I^{-1}(x) \cap Y$ equals $\pi_I^{-1}(x)$ or it is empty.
Moreover, \cite[Theorem 5.9]{Dim:quotroot} implies that $\Phi_1$ and $\Phi_4$ are both inflated from $I$.
Therefore, since $\Phi_1$ and $\Phi_4$ are disjoint, the second condition of the lemma holds for the chosen $I$.
\end{proof}

   We are now ready to complete the proof of Theorem \ref{th:combi}. Let $I$ as in Lemma~\ref{lem: inflation}. 
    By contradiction, let $\varphi$ be a root satisfying $\beta \leq \varphi$ and $\varphi \in \Phi_2$. 
    It follows that $\beta \in \Phi_I$. 
    Moreover, since $\supp(\gamma)= \Delta= \supp(\gamma + \beta)$, we have that neither $\gamma$ nor $\gamma + \beta$ belongs to $\Phi_I$. 
  Then, the fact that $\pi_I(\gamma)=\pi_I(\gamma + \beta)$ and Lemma \ref{lem:thetax} contradict the second assertion of Lemma~\ref{lem: inflation}.

\section{The proofs of the corollaries}
\label{sec:pf coro}

\begin{proof}[Proof of Corollary~\ref{cor:Bruhatorder}.]
  Theorem~\ref{th:main} implies that
  $[X_{w_1}]\cdot[X_{w_2}]\cdot [X_{w_3}]=[pt]$.
  On the other hand $w_1\inv X_{w_1}$,  $w_2\inv
  X_{w_2}$ and $w_3\inv X_{w_3}$ intersect transversally
  at the point $\uo$. It follows that

  \begin{equation}
w_1\inv X_{w_1}\cap w_2\inv
  X_{w_2}\cap w_3\inv X_{w_3}=\{\uo\}.\label{eq:inter}
\end{equation}
  Given $x\in W$, we have
  $x\in w_1\inv X_{w_1}$ if and only if $w_1 x\leq w_1$.
  Then the statement translates the fact that $e\uo$ is the only point
   in the intersection~\eqref{eq:inter}.
\end{proof}

\begin{proof}[Proof of Corollary~\ref{cor:D}.]

Let $E_1, \dots, E_s$ be the irreducible components of $X - \Omega$ of codimension one in $X$. 
The same argument of the proof of Proposition~\ref{prop:divPoincare} implies that $s=2 \rk(G)$, and that  
exactly $3\rk(G)$ irreducible components of $Y - Y^\circ$ 
of codimension one are not contracted by $\eta$. 
Because of Proposition~\ref{prop:rameta} there are exactly 
$\sum_i d(w_i^\vee)$ such divisors.


\end{proof}

\begin{proof}[Proof of Corollary~\ref{cor:regface}.]
  The fact that ${\mathcal F}_{(w_1,w_2,w_3)}$ is a regular face of the cone
  $\cLR(G)$ is a direct application of \cite[Theorem~D]{GITEigen}. 
Idem for the fact that any regular face of dimension $2\rk(G)$ is
obtained in such a way.
The fact that any integral point in ${\mathcal F}_{(w_1,w_2,w_3)}$ belongs to
the semigroup $\LR(G)$ is a consequence of the PRV conjecture (see
\cite{Kumar:prv1,Mathieu:prv} or \cite{MPR}).

Let $E_1,\dots,E_{2\rk(G)}$ be the irreducible components of
$X-\Omega$ of codimension one in $X$.
Each $E_i$ gives a line bundle $\Oc(E_i)$ on $X$ and a section
$\sigma_i\in \Ho^0(X,\Oc(E_i))$ such that $\div(\sigma_i)=E_i$.
Since $G$ is semisimple and simply connected, $\Oc(E_i)$ admits a
unique $G$-linearisation.
Thus, there exists $(\lambda_1^i,\lambda_2^i,\lambda_3^i)\in X(T)^3$
such that $\Oc(E_i)=\Li(\lambda_1^i,\lambda_2^i,\lambda_3^i)$.
Since $E_i$ is $G$-stable and $G$ has no character, $\sigma_i$ is
$G$-invariant. 
Then, $(\lambda_1^i,\lambda_2^i,\lambda_3^i)$ belongs to $\LR(G)$.

By construction $\sigma_i([e:x_0])\neq 0$, where $x_0=(w_1\inv \uo, w_2\inv \uo, w_3\inv \uo)$.
Since $[e:x_0]$ is fixed by the maximal torus $T$ and $\sigma_i$ is
$G$-invariant, $T$ has to act trivially on the fiber in $\Oc(E_i)$
over $[e:x_0]$. 
Thus, $w_1\inv\lambda_1^i+w_2\inv\lambda_2^i+w_2\inv\lambda_3^i=0$ and 
$(\lambda_1^i,\lambda_2^i,\lambda_3^i)$ belongs to
${\mathcal F}_{(w_1,w_2,w_3)}$.

Conversely, let $(\lambda_1,\lambda_2,\lambda_3)$ in
${\mathcal F}_{(w_1,w_2,w_3)}$. Set $\Li=\Li(\lambda_1,\lambda_2,\lambda_3)$.
By {\it e.g.} \cite[Theorem~1.2]{reduction}, the restriction map
induces an isomorphism from $\Ho(X,\Li)^G$ onto
$\Ho^0(\{x_0\},\Li)^T\simeq\CC$.
 Fix a nonzero element $\sigma$ in $\Ho(X,\Li)^G$.
From $\sigma(x_0)\neq 0$, one easily deduces that $\sigma$ does not
vanish on $\Omega=\eta(G\times_B C^+)$ using $G$-invariance and continuity. 
Then there exist nonnegative integers $n_i$ such that
$\div(\sigma)=\sum_{i=1}^{2\rk(G)} n_i E_i$.
In particular $\Li=\sum_i n_i\Oc(E_i)$ and
$(\lambda_1,\lambda_2,\lambda_3)$ belongs to the semigroup generated
by the triples $(\lambda_1^i,\lambda_2^i,\lambda_3^i)$.
\end{proof}

\section{A determinant}
\label{sec: A determinant}

Fix a poset $\{\varphi_0,\dots,\varphi_k\}$ numbered in such a way
that 
$\varphi_i\leq\varphi_j$ only if $i\leq j$. 
Let $M$ be a $(k\times k)$-matrix whose rows are labeled by
$(\varphi_0,\dots,\varphi_{k-1})$ and  columns by
$(\varphi_1,\dots,\varphi_{k})$. 
Denote by $m_{ij}$ the entry at row $\varphi_i$ and column $\varphi_j$. 
We assume that
\begin{enumerate}
\item for any $i=1,\dots,k-1$, $m_{ii}=1$; and
\item $m_{ij}\neq 0$ implies $\varphi_i\leq\varphi_j$.
\end{enumerate}

\begin{prop}
  \label{prop:calculdet}
With above notation, the determinant of $M$ is
$$
\det M=(-1)^{k+1}\sum_{
  \begin{array}{c}
0\leq s\leq k-1\\
0< j_0<\cdots<j_s<k\\
\varphi_0<\varphi_{j_0}<\cdots<\varphi_{j_s}<\varphi_{k}
  \end{array}
}
(-1)^s m_{0\,j_0}m_{j_0\,j_1}\cdots m_{j_s\,k}.
$$
\end{prop}

\begin{proof}
  Start with the expression

  \begin{equation}
\det(M)=\sum_\sigma \varepsilon(\sigma) m_{\sigma(1)\,1}\dots
m_{\sigma(k)\,k},
\label{eq:det2}
\end{equation}
where the sum runs over all the bijections
$\sigma\;:\;[1;k]\longto[0;k-1]$.
Here, $\varepsilon(\sigma)$ is the signature of the
bijection $\tilde \sigma$ of $[1;k]$ on itself that maps $j$ to $\sigma(j)+1$.

Since $M$ is ``almost upper triangular'', in the sum~\eqref{eq:det2},
we can keep only the bijections $\sigma$ satisfying:  $\sigma(j)\leq j$
for any $j\in [1;k]$.
Define $j_0$ by $\sigma(j_0)=0$.
Then, for any $1\leq j<j_0$, we have $\sigma(j)=j$.
In other words, the bijection $\tilde\sigma$ stabilizes $[1;j_0]$ and
acts on it as the cycle $(1,2,\dots,j_0)$.
Moreover,
$$
m_{\sigma(1)\,1}\dots m_{\sigma(j_0)\,j_0}=m_{0\,j_0},
$$
since all the other factors are of the form $m_{jj}=1$.

Now, an immediate induction shows that the expression of $\tilde
\sigma$ as a product of disjoint  cycles can be obtained by
bracketing the word $1\;2\;\dots k$.
Write
$$
\tilde\sigma=(1,\;2,\dots,j_0)(j_0+1,\dots,j_1)\cdots (j_s+1,\dots,k)
$$
allowing cycles of length 1. Then the product, associated to $\sigma$ in \eqref{eq:det2},
is
$$
m_{0\,j_0}m_{j_0\,j_1}\dots m_{j_{s-1}j_{s}}  m_{j_s\, {k}},
$$
and, the signature $\varepsilon(\sigma)$ is
$(-1)^{(j_0-1)+(j_1-j_0-1)+\cdots+(k-j_{s}-1)}=(-1)^{k+s+1}$.
The proposition follows.
\end{proof}

\begin{remark}
  {\it A determinantal expression of the Möbius function $\mu$ for finite
    posets.}
Let $(P,\leq)$ be a finite poset and $[\varphi_0;\varphi_k]=\{\varphi_0,\dots,\varphi_k\}$ be an
interval of $P$.
Let $M$ be the $(k\times k)$-matrix whose rows are labeled by
$[0,k-1]$ and  columns by
$[1,k]$ defined by  
$$
m_{ij} =\left\{
  \begin{array}{lll}
    1&{\rm\ if\ }&\varphi_i\leq\varphi_j,\\
0&{\rm\ otherwise}.
  \end{array}
\right . $$
The proof of Proposition~\ref{prop:calculdet} shows that
$$
\mu([\varphi_0;\varphi_k])=(-1)^k\det M.
$$
The authors do not know if this formula was known before.
\end{remark}


\bibliographystyle{alpha}

\bibliography{bkgb}

\begin{center}
  -\hspace{1em}$\diamondsuit$\hspace{1em}-
\end{center}

\end{document}